\documentclass{article}%
\usepackage{amssymb}
\usepackage{amsfonts}
\usepackage{amsmath}
\usepackage{graphicx}%
\providecommand{\U}[1]{\protect\rule{.1in}{.1in}}
\newtheorem{theorem}{Theorem}

\newtheorem{corollary}[theorem]{Corollary}

\newtheorem{lemma}[theorem]{Lemma}

\newtheorem{remark}[theorem]{Remark}

\newenvironment{proof}[1][Proof]{\noindent\textbf{#1.} }{\ \rule{0.5em}{0.5em}}
\begin{document}

\title{$\mathbb{R}$-orbit Reflexive Operators}
\author{Don Hadwin\\University of New Hampshire
\and Ileana Iona\c{s}cu\\Philadelphia University
\and Hassan Yousefi\\California State University Fullerton}
\maketitle

\begin{abstract}
We completely characterize orbit reflexivity and $\mathbb{R}$-orbit
reflexivity for matrices in $\mathcal{M}_{N}\left(  \mathbb{R}\right)  $.
Unlike the complex case in which every matrix is orbit reflexive and
$\mathbb{C}$-orbit reflexivity is characterized solely in terms of the Jordan
form, the orbit reflexivity and $\mathbb{R}$-orbit reflexivity of a matrix in
$\mathcal{M}_{N}\left(  \mathbb{R}\right)  $ is described in terms of the
linear dependence over $\mathbb{Q}$ of certain elements of $\mathbb{R}%
/\mathbb{Q}$. We also show that every $n\times n$ matrix over an uncountable
field $\mathbb{F}$ is algebraically $\mathbb{F}$-orbit reflexive.

\end{abstract}

\section{Introduction}

The term \emph{reflexive operator} was coined by P. R. Halmos \cite{PRH}, and
studied by many authors, e.g., \cite{Ar}, \cite{Ar2}, \cite{AP}, \cite{C},
\cite{D}, \cite{Dedd}, \cite{DF}, \cite{Fil}, \cite{H1}, \cite{H3},
\cite{HK1}, \cite{HK2}, \cite{HL}, \cite{Han}, \cite{J}, \cite{Kad},
\cite{Lar}, \cite{LS}, \cite{P}, \cite{Sa}, \cite{Sh}. It was proved by J.
Deddens and P. Fillmore \cite{DF} that an $n\times n$ complex matrix $T$ is
reflexive if and only if, for each eigenvalue $\lambda$ of $T$, the two
largest Jordan blocks corresponding to $\lambda$ in the Jordan canonical form
of $T$ differ in size by at most $1$. Later, D. Hadwin \cite{H1} characterized
(algebraic) reflexivity for an $n\times n$ matrix over an arbitrary field; in
this setting the analog of the Jordan form contains blocks, which we will
still call Jordan blocks, of the form
\[
J_{m}\left(  A\right)  =\left(
\begin{array}
[c]{ccccc}%
A & I & 0 & \cdots & 0\\
0 & A & I & \ddots & \vdots\\
0 & 0 & A & \ddots & 0\\
\vdots & \vdots & \ddots & A & I\\
0 & 0 & \cdots & 0 & A
\end{array}
\right)  ,
\]
where $A$ is the companion matrix of an irreducible factor of the minimal
polynomial for $T$. When the irreducible factor has degree $1$, the matrix $A$
is $1\times1$ and an eigenvalue of $T.$ Hadwin \cite{H1} proved that an
$n\times n$ matrix $T$ over a field $\mathbb{F}$ is (algebraically) reflexive
if, for each eigenvalue of $T$, the two largest Jordan blocks differ in size
by at most $1,$ and for an irreducible factor of the minimal polynomial of $T$
that has degree greater than $1$, the two largest Jordan blocks have the same size.

D. Hadwin, E. A. Nordgren, H. Radjavi and P. Rosenthal \cite{HNRR} introduced
the notion of an \emph{orbit-reflexive operator}. They proved that on a
Hilbert space this class includes all normal operators, algebraic operators,
compact operators, contractions and unilateral weighted shift operators. It
was over twenty years before examples were constructed \cite{GR} and \cite{MV}
(see also \cite{Est}) of operators that are not orbit reflexive. V. M\"{u}ller
and J. Vr\v{s}ovsk\'{y} \cite{MV} proved that if $r\left(  T\right)  \neq1$
($r\left(  T\right)  $ denotes the spectral radius of $T$), then $T$ is orbit
reflexive. In \cite{HIY}, where the notion of \emph{null-orbit reflexive
operator} was introduced, the authors proved that every polynomially bounded
operator on a Hilbert space is orbit reflexive.

Recently, M. McHugh and the authors \cite{HIMY}, \cite{Mc} introduced the
notion of $\mathbb{C}$-orbit reflexivity and $\mathbb{R}$-orbit reflexivity,
and they proved that an $n\times n$ complex matrix $T$ is $\mathbb{C}$-orbit
reflexive if and only if it is nilpotent or, among all the Jordan blocks
corresponding to all eigenvalues with modulus equal to the spectral radius
$r\left(  T\right)  $ of $T$, the two largest blocks differ in size by at most
$1$.

In this paper we address orbit reflexivity and $\mathbb{R}$-orbit reflexivity
for a matrix in $\mathcal{M}_{n}\left(  \mathbb{R}\right)  $. In
$\mathcal{M}_{n}\left(  \mathbb{C}\right)  $ every matrix is orbit reflexive
are $\mathbb{C}$-orbit reflexivity is characterized solely in terms of the
Jordan form. Surprisingly, neither of these facts remain true for
$\mathcal{M}_{n}\left(  \mathbb{R}\right)  ;$ the characterizations involve a
little number theory, i.e., linear dependence over $\mathbb{Q}$ of elements in
$\mathbb{R}/\mathbb{Q}$.

\section{Algebraic Results}

An irreducible factor $p\left(  x\right)  $ of a polynomial in $\mathbb{R}%
\left[  x\right]  $ has degree at most $2$. If $p\left(  x\right)
\in\mathbb{R}\left[  x\right]  $ is monic and irreducible and $\deg p=2$, then
$p$ has roots $\alpha\pm i\beta$ with $a,\beta\in\mathbb{R}$, $\beta\neq0,$
$p\left(  x\right)  =\left(  x-\alpha\right)  ^{2}+\beta^{2}$, and the
corresponding companion matrix looks like $\left(
\begin{array}
[c]{cc}%
\alpha & -\beta\\
\beta & \alpha
\end{array}
\right)  =r\left(
\begin{array}
[c]{cc}%
\cos\theta & -\sin\theta\\
\sin\theta & \cos\theta
\end{array}
\right)  ,$ where
\[
\alpha+i\beta=re^{i\theta}%
\]
with $r=\sqrt{\alpha^{2}+\beta^{2}}$ and $0\leq\theta<2\pi$. The matrix
\[
R_{\theta}=\left(
\begin{array}
[c]{cc}%
\cos\theta & -\sin\theta\\
\sin\theta & \cos\theta
\end{array}
\right)
\]
acts on $\mathbb{R}^{2}$ as a counterclockwise rotation by the angle $\theta$.
More generally, if we identify $\mathbb{R}^{2}$ with $\mathbb{C}$, then
$\left(
\begin{array}
[c]{cc}%
\alpha & -\beta\\
\beta & \alpha
\end{array}
\right)  $ acts as multiplication by $\alpha+i\beta$. An $m\times m$ Jordan
block corresponding to $A=\left(
\begin{array}
[c]{cc}%
\alpha & -\beta\\
\beta & \alpha
\end{array}
\right)  $, is given by $J_{m}\left(  A\right)  .$ However, $J_{m}\left(
A\right)  $ is similar to $rJ_{m}\left(  R_{\theta}\right)  $, and we will
represent the Jordan blocks this way. A Jordan block $J$ of $T$ \emph{splits,
}or,\emph{ is splitting}, if the irreducible polynomial associated to it has
degree $1,$ i.e., it corresponds to a real eigenvalue of $T.$

Since a real matrix may have empty spectrum, we let $\sigma_{p}\left(
T\right)  $ denote the \emph{point spectrum} of $T$, the set of real
eigenvalues of $T$. Note that $\sigma_{p}\left(  T\right)  =\varnothing$ is
possible. We define the \emph{spectral radius} to be%
\[
r\left(  T\right)  =\lim_{n\rightarrow\infty}\left\Vert T^{n}\right\Vert
^{\frac{1}{n}},
\]
which is the spectral radius of $T$ considered as a matrix in $\mathcal{M}%
_{n}\left(  \mathbb{C}\right)  $. Note that $r\left(  J_{m}\left(  R_{\theta
}\right)  \right)  =1$ and $r\left(  \left(
\begin{array}
[c]{cc}%
\alpha & -\beta\\
\beta & \alpha
\end{array}
\right)  \right)  =\sqrt{\alpha^{2}+\beta^{2}}$.

If $X$ is a vector space over a field $\mathbb{F}$, and $T$ is a linear
transformation on $X$, then $\mathcal{P}_{\mathbb{F}}\left(  T\right)
=\left\{  p\left(  T\right)  :p\in\mathbb{F}\left[  t\right]  \right\}  $. A
\emph{linear manifold }$M$ in $X$, is the translate of a linear subspace,
i.e., nonempty subset $M$ so that when $x\in M$, $M-x$ is a linear subspace.

We begin with a lemma on the cardinality of the field. In the case where the
field is $\mathbb{R}$ or $\mathbb{C}$, the lemma is an immediate consequence
of the Baire category theorem.

\bigskip

\begin{lemma}
\label{sub}If $\mathbb{F}$ is an uncountable field and $n$ is a positive
integer, then $\mathbb{F}^{n}$ is not a countable union of proper linear subspaces.
\end{lemma}

\begin{proof}
Let $S=\left\{  \left(  1,x,x^{2},\ldots,x^{n-1}\right)  :x\in\mathbb{F}%
\right\}  $. Since any $n$ distinct elements of $S$ are linearly independent,
the intersection of any proper linear subspace with $S$ has cardinality at
most $n-1.$ However, $S$ is uncountable, so $S$ is not contained in a
countable union of proper linear subspaces of $\mathbb{F}^{n}$.
\end{proof}

\bigskip

\begin{theorem}
\label{alg}If $\mathbb{F}$ is an uncountable field, then every $T\in
\mathcal{M}_{N}\left(  \mathbb{F}\right)  $ is algebraically $\mathbb{F}%
$-orbit reflexive and algebraically orbit-reflexive.
\end{theorem}

\begin{proof}
It is known from \cite{HK1} that \textrm{AlgLat}$_{0}\left(  T\right)
\cap\left\{  T\right\}  ^{\prime}=\mathcal{P}_{\mathbb{F}}\left(  T\right)  $,
and that this algebra of operators has a separating vector $e.$ We know from
\cite{HIMY} that every nilpotent matrix is algebraically $\mathbb{F}$-orbit
reflexive. Suppose $A$ is an invertible $k\times k$ matrix and $S\in
\mathbb{F}$\textrm{-}$\mathrm{OrbRef}_{0}\left(  A\right)  .$ Then, for every
$x\in\mathbb{F}^{k}$, there is a $\lambda\in\mathbb{F}$ and an $m\geq0$ such
that $Sx=\lambda A^{m}x$. Hence,
\[
\mathbb{F}^{k}=\bigcup_{m=0}^{\infty}\bigcup_{\lambda\in\sigma_{p}\left(
A^{-m}S\right)  }Ker\left(  A^{-m}S-\lambda\right)  ,
\]
which, by Lemma \ref{sub}, implies there is an $m\geq0$ and a $\lambda
\in\mathbb{F}$ such that $S=\lambda A^{m}$. Hence $A$ is algebraically
$\mathbb{F}$-orbit reflexive. Since every $T\in\mathcal{M}_{n}\left(
\mathbb{F}\right)  $ is the direct sum of a nilpotent matrix $N$ and an
invertible matrix $A$, it follows that every $S\in\mathbb{F}$\textrm{-}%
$\mathrm{\mathrm{OrbRef}}_{0}\left(  T\right)  $ is a direct sum of $\alpha
N^{s}$ and $\beta A^{t}$ for $\alpha,\beta\in\mathbb{F}$ and integers
$s,t\geq0$. It follows that $S\in$\textrm{AlgLat}$_{0}\left(  T\right)
\cap\left\{  T\right\}  ^{\prime};$ whence there is a polynomial
$p\in\mathbb{F}\left[  x\right]  $ such that $S=p\left(  T\right)  $. However,
there is a $\lambda\in\mathbb{F}$ and an $m\geq0$ such that
\[
p\left(  T\right)  e=Se=\lambda T^{m}e.
\]
Since $e$ is separating for $\mathcal{P}\left(  T\right)  ,$ we see that
$S=p\left(  T\right)  =\lambda T^{m},$ which implies $T$ is $\mathbb{F}$-orbit
reflexive. The proof that $T$ is algebraically orbit reflexive is very similar.
\end{proof}

\bigskip

\begin{corollary}
\label{power}If $T\in\mathcal{M}_{n}\left(  \mathbb{R}\right)  $ and $\left\{
T^{k}:k\geq0\right\}  $ is finite, e.g., $T^{N}=I$ or $T^{N}=0$ for some
positive integer $N,$ then $\mathbb{R}$-$\mathrm{OrbRef}\left(  T\right)
=\mathbb{R}$-$\mathrm{OrbRef}_{0}\left(  T\right)  =\mathbb{R}$-$\mathrm{Orb}%
\left(  T\right)  $ and $\mathrm{OrbRef}\left(  T\right)  =\mathrm{OrbRef}%
_{0}\left(  T\right)  =\mathrm{Orb}\left(  T\right)  $.
\end{corollary}

\begin{proof}
Since $\left\{  T^{k}:k\geq0\right\}  $ is finite, we know, for every vector
$x$, that $\mathbb{R}$-$\mathrm{Orb}\left(  T\right)  x$ and $\mathrm{Orb}%
\left(  T\right)  x$ are closed, implying $\mathbb{R}$-$\mathrm{OrbRef}\left(
T\right)  =\mathbb{R}$-$\mathrm{OrbRef}_{0}\left(  T\right)  $ and$\mathbb{\ }%
\mathrm{OrbRef}\left(  T\right)  =\mathrm{OrbRef}_{0}\left(  T\right)  $.
\end{proof}

\bigskip

\begin{corollary}
\label{help}If $T\in\mathcal{M}_{n}\left(  \mathbb{R}\right)  ,$ $T=A\oplus B$
with $A^{N}=I$ for some minimal $N\geq1$ and $r\left(  B\right)  <1$, then $T$
is $\mathbb{R}$-orbit reflexive.
\end{corollary}

\begin{proof}
Suppose $S\in\mathbb{R}$-$\mathrm{OrbRef}\left(  T\right)  $. Then
$S=S_{1}\oplus S_{2}$ and, by Corollary \ref{power}, we know that
$S_{1}=\lambda A^{s}$ for some $\lambda\in\mathbb{R}$ and some $s\geq0$. If
$S_{1}=0$ it easily follows by considering $x\oplus y$ with $x\neq0$ and $y$
arbitrary, that $S_{2}=0$, which implies $S=0.$ Hence we can assume that
$S_{1}\neq0$.

Note that
\[
S_{1}^{N}=\lambda^{N}\left(  A^{N}\right)  ^{s}=\lambda^{N}.
\]
Let $E=\left\{  e^{2\pi ik/n}\lambda:k=1,\ldots,n\right\}  $. Choose a
separating unit vector $x_{0}$ for $\mathcal{P}_{R}\left(  A\right)  .$ If
$S_{1}x_{0}=\lambda_{1}A^{t}x_{0},$ we have $S_{1}=\lambda_{1}A^{t},$ which
implies $\lambda_{1}\in E$. Suppose $y$ is in the domain of $B$, then there is
a sequence $\left\{  k_{m}\right\}  $ of positive integers and a sequence
$\left\{  \beta_{m}\right\}  $ in $\mathbb{R}$ such that
\[
\beta_{m}T^{k_{m}}\left(  x_{0}\oplus y\right)  \rightarrow S_{1}x_{0}\oplus
S_{2}y.
\]
We have $\beta_{m}A^{k_{m}}x_{0}\rightarrow\lambda A^{s}x_{0},$ which implies
$\left\{  \beta_{m}\right\}  $ is bounded. If $\left\{  k_{m}\right\}  $ is
unbounded, then it has a subsequnce diverging to $\infty$, which implies
$S_{2}y=0$, since $\left\Vert B^{k}\right\Vert \rightarrow0$ as $k\rightarrow
\infty$. If $\left\{  k_{m}\right\}  $ is bounded, then it has a subsequence
$\left\{  \beta_{k_{j}}\right\}  $ with a constant value $t$, and we get
$\beta_{m_{j}}\rightarrow\lambda_{1}$ for some $\lambda_{1}\in E$. Hence the
domain of $B$ is a countable union,%
\[
\ker S_{2}\cup\bigcup_{k\in\mathbb{N},\gamma\in E}\ker\left(  S_{2}-\gamma
B^{k}\right)  .
\]
It follows from Lemma \ref{sub} that $S_{2}\in\mathcal{P}_{\mathbb{R}}\left(
B\right)  .$ If we choose a vector $y_{0}$ that is separating for
$\mathcal{P}_{\mathbb{R}}\left(  B\right)  ,$ we see from $S\left(
x_{0}\oplus y_{0}\right)  \in\left[  \mathbb{R}\text{-}\mathrm{Orb(T)}\left(
x_{0}\oplus y_{0}\right)  \right]  ^{-},$ that $S\in\mathbb{R}$%
-$\mathrm{Orb(T).}$
\end{proof}

\bigskip

\section{Main Results}

A key ingredient in our proofs is the following well-known result from number
theory. We sketch the elementary proof for completeness. For notation we let
$\mathbb{T}=\left\{  z\in\mathbb{C}\mathbf{:}\left\vert z\right\vert
=1\right\}  $ be the unit circle, $\mathbb{T}^{k}$ a direct product of $k$
copies of $\mathbb{T}$, and $\mu_{k}=\mu\times\cdots\times\mu$ be Haar measure
on $\mathbb{T}^{k}$, where $\mu$ is normalized arc length on $\mathbb{T}$. If
$\lambda=\left(  z_{1},\cdots,z_{k}\right)  \in\mathbb{T}^{k}$ and we define%
\[
\lambda^{n}=\left(  z_{1}^{n},\ldots,z_{k}^{n}\right)
\]
for $n=0,1,2,\ldots$ .

\bigskip

\begin{lemma}
\label{ind}Suppose $\theta_{1},\ldots,\theta_{k}\in\mathbb{R}$, and let
$\lambda=\left(  e^{i\theta_{1}},\ldots,e^{i\theta_{k}}\right)  $. The
following are equivalent:

\begin{enumerate}
\item $\left\{  \lambda,\lambda^{2},\ldots\right\}  $ is dense in
$\mathbb{T}^{k}$,

\item $\left\{  1,\theta_{1}/2\pi,\ldots,\theta_{k}/2\pi\right\}  $ is
linearly independent over $\mathbb{Q}$,

\item for every $f\in C\left(  \mathbb{T}^{k}\right)  $ we have%
\[
\lim_{N\rightarrow\infty}\frac{1}{N}%
%TCIMACRO{\dsum _{n=1}^{N}}%
%BeginExpansion
{\displaystyle\sum_{n=1}^{N}}
%EndExpansion
f\left(  \lambda^{n}\right)  =\int_{\mathbb{T}^{k}}fd\mu_{k}.
\]

\end{enumerate}
\end{lemma}

\begin{proof}
If $f\left(  z_{1},\ldots,z_{k}\right)  =z_{1}^{m_{1}}\cdots z_{k}^{m_{k}}$
for integers $m_{1},\ldots,m_{k}$, then statement $\left(  2\right)  $ is
equivalent to saying $f\left(  \lambda\right)  \neq1$ whenever $\left(
m_{1},\ldots,m_{k}\right)  \neq\left(  0,\ldots,0\right)  $. For such a
monomial $f$ we know that $\int_{\mathbb{T}^{k}}fd\mu_{k}=0,$ and we know that
$f\left(  \lambda^{n}\right)  =f\left(  \lambda\right)  ^{n}$ for $n\geq1.$
Thus statement $\left(  2\right)  $ implies that%
\[
\lim_{N\rightarrow\infty}\frac{1}{N}%
%TCIMACRO{\dsum _{n=1}^{N}}%
%BeginExpansion
{\displaystyle\sum_{n=1}^{N}}
%EndExpansion
f\left(  \lambda^{n}\right)  =\lim_{N\rightarrow\infty}\frac{1}{N}%
\frac{1-f\left(  \lambda\right)  ^{N}}{1-f\left(  \lambda\right)  }f\left(
\lambda\right)  \rightarrow0=\int_{\mathbb{T}^{k}}fd\mu_{k}.
\]
It follows from the Stone-Weierstrass theorem that the span of the monomials
is dense in $C\left(  \mathbb{T}^{k}\right)  ,$ so we see that $\left(
2\right)  \Longrightarrow\left(  3\right)  .$ On the other hand $\left(
3\right)  $ implies that, for every nonnegative continuous function $f$
vanishing on $\left\{  \lambda,\lambda^{2},\ldots\right\}  $ we must have
$\int_{\mathbb{T}^{k}}fd\mu_{k}=0$, which implies $f=0$. If $x\in
\mathbb{T}^{k}\backslash\left\{  \lambda,\lambda^{2},\ldots\right\}  ^{-},$
there is a nonnegative continuous function $f$ vanishing on $\left\{
\lambda,\lambda^{2},\ldots\right\}  $ with $f\left(  x\right)  \neq0$. Hence
$\left(  3\right)  \Longrightarrow\left(  1\right)  $. If $f$ is a nonconstant
monomial and $f\left(  \lambda\right)  =1,$ then the closure of $\left\{
\lambda,\lambda^{2},\ldots\right\}  $ is contained in $f^{-1}\left(  \left\{
1\right\}  \right)  ,$ which proves that $\left(  1\right)  \Longrightarrow
\left(  2\right)  $.\bigskip
\end{proof}

The next two results show that in $\mathcal{M}_{N}\left(  \mathbb{R}\right)  $
orbit reflexivity is not the same as in $\mathcal{M}_{N}\left(  \mathbb{C}%
\right)  $.

\begin{lemma}
\label{hard}Suppose $k\in\mathbb{N}$, $\theta_{1},\ldots,\theta_{k}\in
\lbrack0,2\pi),$ and $T\in\mathcal{M}_{N}\left(  \mathbb{R}\right)  $ is a
direct sum of $R_{\theta_{1}}\oplus\cdots\oplus R_{\theta_{k}}\oplus B\oplus
C$ with $B^{2}=1$ and $r\left(  C\right)  <$ $1$. (The summands $B$ and $C$
might not be present.) The following are equivalent:

\begin{enumerate}
\item $T$ is orbit reflexive

\item $T$ is $\mathbb{R}$-orbit reflexive

\item There are nonzero integers $s_{1},\ldots,s_{k}$ and an integer $t$ such
that
\[%
%TCIMACRO{\dsum _{j=1}^{k}}%
%BeginExpansion
{\displaystyle\sum_{j=1}^{k}}
%EndExpansion
s_{j}\theta_{j}=2\pi t.
\]

\item For every $j\in\left\{  1,\ldots,k\right\}  ,$ $\theta_{j}/2\pi\in
sp_{\mathbb{Q}}\left(  \left\{  1\right\}  \cup\left\{  \theta_{i}/2\pi:1\leq
i\neq j\leq k\right\}  \right)  .$
\end{enumerate}
\end{lemma}

\begin{proof}
The equivalence of $\left(  4\right)  $ and $\left(  3\right)  $ is easy.

$\left(  1\right)  \Longrightarrow\left(  4\right)  $ and $\left(  2\right)
\Longrightarrow\left(  4\right)  .$ Assume $\left(  4\right)  $ is false. We
can assume that
\[
\theta_{1}/2\pi\notin sp_{\mathbb{Q}}\left(  \left\{  1\right\}  \cup\left\{
\theta_{i}/2\pi:2\leq i\leq k\right\}  \right)  .
\]
We can assume that $\left\{  1,\theta_{2}/2\pi,\ldots,\theta s/2\pi\right\}  $
is a basis for the linear span over $\mathbb{Q}$ of $\left\{  1\right\}
\cup\left\{  \theta_{i}/2\pi:2\leq i\leq k\right\}  ,$ which makes $\theta
_{1}/2\pi,\theta_{2}/2\pi,\ldots,\theta_{s}/2\pi$ irrational, and makes
$\left\{  1,\theta_{1}/2\pi,\ldots,\theta s/2\pi\right\}  $ linearly
independent over $\mathbb{Q}$. Since each $\theta_{j}/2\pi$, $s<j\leq k$ is a
rational linear combination of $1,\theta_{2}/2\pi,\ldots,\theta s/2\pi$, there
is a positive integer $d$ such that, for $s<j\leq k$, each $d\theta_{j}/2\pi$
is an integral linear combination of $1,\theta_{2}/2\pi,\ldots,\theta s/2\pi.$
Suppose $\alpha\in\lbrack0,2\pi)$. Since $\left\{  1,\theta_{1}/4\pi
d,\ldots,\theta s/4\pi d\right\}  $ is linearly independent over $\mathbb{Q}$,
it follows from Lemma \ref{ind} that there is a sequence $\left\{
m_{n}\right\}  $ of positive integers such that $m_{n}\rightarrow\infty,$
\[
R_{\theta_{1}}^{m_{n}}=R_{m_{n}\theta_{1}}\rightarrow R_{\alpha/2d},
\]%
\[
R_{\theta_{j}}^{m_{n}}=R_{m_{n}\theta_{j}}\rightarrow I
\]
for $2\leq j\leq s.$ This implies that $R_{\theta_{1}}^{2dm_{n}}%
=R_{2dm_{n}\theta_{1}}\rightarrow R_{\alpha}$ and $R_{\theta j}^{2dm_{n}%
}=R_{2dm_{n}\theta_{j}}\rightarrow I$ for $2\leq j\leq s.$ If $s<j\leq k$,
there are integers $t_{2},\ldots,t_{s}$ and $t$ such that $d\theta_{j}=t2\pi+%
%TCIMACRO{\dsum _{i=2}^{s}}%
%BeginExpansion
{\displaystyle\sum_{i=2}^{s}}
%EndExpansion
t_{i}\theta_{i},$ which implies%
\[
R_{\theta j}^{2dm_{n}}=I^{2tm_{n}}%
%TCIMACRO{\dprod _{i=2}^{s}}%
%BeginExpansion
{\displaystyle\prod_{i=2}^{s}}
%EndExpansion
\left(  R_{m_{n}\theta_{i}}\right)  ^{2t_{i}}\rightarrow I.
\]
Moreover,
\[
\left(  B\oplus C\right)  ^{2dm_{n}}=B^{2dm_{n}}\oplus C^{2dm_{n}}\rightarrow
I\oplus0=P.
\]
Let $F=\left(
\begin{array}
[c]{cc}%
0 & 1\\
1 & 0
\end{array}
\right)  ,$ and define $S=F\oplus I\oplus\cdots\oplus I\oplus P$. It follows
from the fact that, for every $x\in\mathbb{R}^{2}$ there is an $\alpha
\in\lbrack0,2\pi)$ such that $Fx=R_{\alpha}x,$ that $S\in\mathrm{OrbRef}%
\left(  T\right)  \subseteq\mathbb{R}$-$\mathrm{OrbRef}\left(  T\right)  $.
Since $FR_{\theta_{1}}\neq R_{\theta_{1}}F$ (because $\sin\theta_{1}\neq0$),
it follows that $ST\neq TS,$ and we see that both $\left(  1\right)  $ and
$\left(  2\right)  $ are false.

$\left(  3\right)  \Longrightarrow\left(  2\right)  .$ Suppose $\left(
3\right)  $ is true. If $k=1,$ then $\theta_{1}/2\pi\in\mathbb{Q}$, and
$R_{\theta_{1}}^{N}=I$ for some positive integer $N,$ which, by Corollary
\ref{help}, implies $T$ is $\mathbb{R}$-orbit reflexive. Hence we can assume
$k\geq2$, which, by $\left(  3\right)  $, implies $\theta_{1}/2\pi
\notin\mathbb{Q}$. Suppose $S\in\mathbb{R}$-$\mathrm{OrbRef}\left(  T\right)
.$ Since $\mathbb{R}$-$\mathrm{OrbRef}\left(  T\right)  $ is contained in
\textrm{AlgLat}$\left(  T\right)  ,$ we can write $S=S_{1}\oplus\cdots\oplus
S_{k}\oplus$ $D\oplus E$. Suppose $x\neq0$ is in the domain of $S_{1}$. We
consider two cases:

\textbf{Case 1}. $S_{1}x=0.$ If $y$ is any vector orthogonal to the domain of
$S_{1}$, there is a sequence $\left\{  m_{n}\right\}  $ of nonnegative
integers and a sequence $\left\{  \lambda_{n}\right\}  $ in $\mathbb{R}$ such
that $S\left(  x\oplus y\right)  =\lim\lambda_{n}T^{m_{n}}\left(  x\oplus
y\right)  .$ Thus $\left\vert \lambda_{n}\right\vert \left\Vert x\right\Vert
\rightarrow\left\Vert S_{1}x\right\Vert =0,$ which implies $\lambda
_{n}\rightarrow0,$ and since $\left\{  \left\Vert T^{n}\right\Vert \right\}  $
is bounded, we see that $S\left(  x\oplus y\right)  =0.$ Thus $0=S_{2}%
=\cdots=S_{k}$ and $D=0,E=0.$ Since $k\geq2,$ and arguing as above (when we
showed $S_{1}=0\Longrightarrow S_{2}=0$), we know $S_{1}=0,$ and thus $S=0.$

\textbf{Case 2}. $S_{1}x\neq0$. Let $x_{1}=x,$ and choose $x_{j}$ in the
domain of $S_{j}$ for $2\leq j\leq k$ with each $\left\Vert x_{j}\right\Vert
=\left\Vert x\right\Vert ,$ and let $u=x\oplus x_{2}\oplus\cdots\oplus
x_{k}\oplus0\oplus0$. Since $R_{\theta_{1}}\oplus\cdots\oplus R_{\theta_{k}}$
is an isometry and $S\in\mathbb{R}$-$\mathrm{OrbRef}\left(  T\right)  ,$ it
follows that there is a sequence $\left\{  m_{n}\right\}  $ of nonnegative
integers and a sequence $\left\{  \lambda_{n}\right\}  $ in $\mathbb{R}$ such
that $0\neq Su=\lim_{n\rightarrow\infty}\lambda_{n}T^{m_{n}}u.$ Hence,
$\left\{  \lambda_{n}\right\}  $ is bounded, so we can assume that
$\lambda_{n}\rightarrow\lambda$ for some nonzero $\lambda\in\mathbb{R}$, and
we can assume that $T^{m_{n}}\rightarrow R_{\alpha_{1}}\oplus\cdots\oplus
R_{\alpha_{k}}\oplus F\oplus G$ with $0\leq\alpha_{1},\ldots,\alpha_{k}<2\pi$.
We know that $\left\vert \lambda\right\vert =\left\Vert S_{1}x\right\Vert
\neq0$, and, for $1\leq j\leq k,$ $S_{j}x_{j}=\left\Vert S_{1}x\right\Vert
R_{\alpha_{j}}x_{j}$ if $\lambda>0$ and $S_{j}x_{j}=\left\Vert S_{1}%
x\right\Vert R_{\alpha_{j}+\pi}x_{j}$ if $\lambda<0$. Moreover, since
$R_{\theta_{j}}^{m_{n}}\rightarrow R_{\alpha_{j}}$ for $1\leq j\leq k,$ we
have, from $\left(  3\right)  $, that $%
%TCIMACRO{\dsum _{j=1}^{k}}%
%BeginExpansion
{\displaystyle\sum_{j=1}^{k}}
%EndExpansion
s_{j}\alpha_{j}\in2\pi\mathbb{Z}$, and thus $%
%TCIMACRO{\dsum _{j=1}^{k}}%
%BeginExpansion
{\displaystyle\sum_{j=1}^{k}}
%EndExpansion
s_{j}\left(  \alpha_{j}+\pi\right)  \in\pi\mathbb{Z}$. Suppose now we replace
$x_{1}$ with another vector $y$ in the domain of $S_{1}$ with $\left\Vert
y\right\Vert =\left\Vert x_{1}\right\Vert ,$ we get real numbers $\beta
_{1},\ldots,\beta_{k}$ such that $S_{1}y=\left\Vert S_{1}y\right\Vert
R_{\beta_{1}}y$ and $S_{j}x_{j}=\left\Vert S_{1}y\right\Vert R_{\beta_{j}%
}x_{j}=\left\Vert S_{1}y\right\Vert R_{\alpha_{j}}x_{j}$ for $2\leq j\leq k,$
and such that $%
%TCIMACRO{\dsum _{j=1}^{k}}%
%BeginExpansion
{\displaystyle\sum_{j=1}^{k}}
%EndExpansion
s_{j}\beta_{j}\in\pi\mathbb{Z}$. However, for $2\leq j\leq k,$ we must have
$\beta_{j}-\alpha_{j}\in\pi\mathbb{Z}$. Hence, $s_{1}\beta_{1}-s_{1}\alpha
_{1}\in\pi\mathbb{Z}$. Hence the domain of $S_{1}$ is the union%
\[
\bigcup_{n\in\mathbb{Z}}\ker\left(  S_{1}-\left\Vert S_{1}x\right\Vert
R_{\alpha_{1}+n\pi/s_{1}}\right)  ,
\]
which, by Lemma \ref{sub}, implies that there is a $\gamma_{1}\in\lbrack
0,2\pi)\cap\left(  \alpha_{1}+\frac{\pi}{s_{1}}\mathbb{Z}+2\pi\mathbb{Z}%
\right)  $ such that $S_{1}=\left\Vert S_{1}x\right\Vert R_{\gamma_{1}}$.
Similarly, we get, for $2\leq j\leq k,$ that $S_{j}=\left\Vert S_{1}%
x\right\Vert R_{\gamma_{j}}$ for some $\gamma_{j}\in\lbrack0,2\pi)$.

Applying the same reasoning we see that $D=\left\Vert S_{1}x\right\Vert B$ or
$D=-\left\Vert S_{1}x\right\Vert B$. Also, for every $f$ in the domain of $C$
we get $Ef\in\mathbb{R}$-$\mathrm{Orb}\left(  C\right)  f,$ so, by Theorem
\ref{alg}, $E\in\mathbb{R}$-$\mathrm{Orb}\left(  C\right)  .$ We therefore
have $S_{j}\in\mathcal{P}_{\mathbb{R}}\left(  R_{\theta_{j}}\right)  $ for
$1\leq j\leq k,$ $D\in\mathcal{P}_{\mathbb{R}}\left(  B\right)  ,$ and
$E\in\mathcal{P}_{\mathbb{R}}\left(  C\right)  .$ If we choose separating
vectors $v_{j}$ for each $\mathcal{P}_{\mathbb{R}}\left(  R_{\theta_{j}%
}\right)  $ $\left(  1\leq j\leq k\right)  $ and $w_{1}$ for $\mathcal{P}%
_{\mathbb{R}}\left(  B\right)  $ and $w_{2}$ for $\mathcal{P}_{\mathbb{R}%
}\left(  C\right)  ,$ and we let $\eta=v_{1}\oplus\cdots\oplus v_{k}\oplus
w_{1}\oplus w_{2}$, then there is a sequence $\left\{  q_{n}\right\}  $ of
nonnegative integers and a sequence $\left\{  t_{n}\right\}  $ in $\mathbb{R}$
such that
\[
t_{n}T^{q_{n}}\eta\rightarrow S\eta,
\]
and it follows that
\[
t_{n}T^{q_{n}}\rightarrow S.
\]
Thus $S\in\mathbb{R}$-$\mathrm{Orb}\left(  T\right)  ^{-SOT}.$

$\left(  2\right)  \Longrightarrow\left(  1\right)  $. Suppose $\left(
2\right)  $ is true, let $e$ be a separating vector for $\mathcal{P}%
_{\mathbb{R}}\left(  T\right)  $, and suppose $S\in\mathrm{\mathrm{OrbRef}%
}\left(  T\right)  \subseteq\mathbb{R}$-$\mathrm{OrbRef}\left(  T\right)
=\mathbb{R}$-$\mathrm{Orb}\left(  T\right)  \subseteq\mathcal{P}_{\mathbb{R}%
}\left(  T\right)  $ (by $\left(  2\right)  $). Since there is a sequence
$\left\{  m_{n}\right\}  $ of nonnegative integers such that $T^{m_{n}%
}e\rightarrow Se$, it follows that $T^{m_{n}}\rightarrow S$. Hence $\left(
1\right)  $ is proved.
\end{proof}

\bigskip

\begin{theorem}
\label{or}A matrix $T\in\mathcal{M}_{N}\left(  \mathbb{R}\right)  $ fails to
be orbit reflexive if and only if it is similar to a matrix of the form in
Lemma \ref{hard} that is not orbit reflexive.
\end{theorem}

\begin{proof}
We know from \cite[Lemma 17]{HIY} that if one of the sets $\left\{
x\in\mathbb{R}^{N}:T^{k}x\rightarrow0\right\}  $ or $\left\{  x\in
\mathbb{R}^{N}:\left\Vert T^{k}x\right\Vert \rightarrow\infty\right\}  $ is
not a countable union of nowhere dense subsets of $\mathbb{R}^{N},$ then $T$
is orbit reflexive. Thus if $r\left(  T\right)  <1,$ then $T$ is orbit
reflexive. If $r\left(  T\right)  >1,$ then the Jordan form shows that
$\left\{  x\in\mathbb{R}^{N}:\left\Vert T^{k}x\right\Vert \rightarrow
\infty\right\}  $ has nonempty interior, which implies $T$ is orbit reflexive.
Hence we are left with the case where $r\left(  T\right)  =1$. Moreover, if
the Jordan form of $T$ has an $m\times m$ block of the form $\left(
\begin{array}
[c]{cccc}%
A & I_{2} & \cdots & 0\\
0 & A & \ddots & \vdots\\
\vdots & 0 & \ddots & I_{2}\\
0 & \cdots & 0 & A
\end{array}
\right)  $ with $A=\pm I$ or $A=R_{\theta}$, then for any vector
$x\in\mathbb{R}^{N}$ whose $m^{th}$-coordinate relative to this summand is
nonzero, we have $\left\Vert T^{k}x\right\Vert \rightarrow\infty;$ whence $T$
is orbit reflexive. Thus the Jordan form of a matrix that is not orbit
reflexive must be as the matrix in Lemma \ref{hard}.
\end{proof}

\bigskip

If $X$ is a Banach space over $\mathbb{R}$, and $T\in B\left(  \mathbb{R}%
\right)  $ is algebraic, i.e., there is a nonzero polynomial $p\in
\mathbb{R}\left[  x\right]  $ such that $p\left(  T\right)  =0,$ then, as a
linear transformation, $T$ has a Jordan form with finitely many distinct
blocks, but possibly with some of the blocks having infinite multiplicity.

\bigskip

\begin{corollary}
\label{coror}Suppose $X$ is a Banach space over $\mathbb{R}$ and $T\in
B\left(  X\right)  $ is algebraic. Then $T$ fails to be orbit-reflexive if and
only if $r\left(  T\right)  =1,$ and the Jordan form for $T$ has one block
$R_{\theta_{1}}$ of multiplicity $1$, other blocks of the form $R_{\theta_{2}%
},\ldots,R_{\theta_{k}}$ with $\theta_{1}/2\pi\notin sp_{\mathbb{Q}}\left\{
1,\theta_{2},\ldots,\theta_{k}\right\}  $, the remaining blocks of the form
$\pm I$ or blocks with spectral radius less than $1$.
\end{corollary}

\begin{proof}
Suppose $T$ has the indicated form.Then there is an invertible operator $D\in
B\left(  X\right)  $ such that $D^{-1}TD=R_{\theta_{1}}\oplus A\oplus B$ with
$r\left(  A\right)  =1$ and $r\left(  B\right)  <1.$ Let $S=F\oplus1\oplus0$.
Suppose $x\in X$. Choose a finite-dimensional invariant subspace $M$ for $T$
of the form $M=M_{1}\oplus M_{2}\oplus M_{3}$, with $M_{1}$ equal to the
domain of $S_{1}$ such that $x\in M$. It follows from the assumptions on $T$
and the proof of Theorem \ref{or} that $S|M\in\mathrm{OrbRef}\left(
T|M\right)  .$ In particular, $Sx$ is in the closure $\mathrm{Orb}\left(
T\right)  x$. Thus $S\in\mathrm{OrbRef}\left(  T\right)  ,$ but $ST\neq TS,$
so $T$ is not orbit reflexive.

On the other hand, if $T$ does not have the described form, then, given
$S\in\mathrm{OrbRef}\left(  T\right)  $, vectors $x_{1},\ldots,x_{n}$ and
$\varepsilon>0,$ there is a finite-dimensional invariant subspace $E$ of $X$
containing $x_{1},\ldots,x_{n}$ such that $T|E$ is orbit reflexive because of
the conditions in Theorem \ref{or}. Hence, since $S|E\in\mathrm{OrbRef}\left(
T|E\right)  $, there is an integer $m\geq0$ such that
\[
\left\Vert Sx_{j}-T^{m}x_{j}\right\Vert <\varepsilon
\]
for $1\leq j\leq n$. Thus $S$ is in the strong operator closure of
$\mathrm{Orb}\left(  T\right)  .$ Thus $T$ is orbit reflexive.
\end{proof}

\bigskip

\begin{theorem}
A matrix $T\in\mathcal{M}_{N}\left(  \mathbb{R}\right)  $ fails to be
$\mathbb{R}$-orbit reflexive if and only if $r\left(  T\right)  \neq0$ with
the largest size of a Jordan block with spectral radius $r\left(  T\right)  $
being $m,$ and either

\begin{enumerate}
\item every Jordan block of $T$ with spectral radius $r\left(  T\right)  $
splits over $\mathbb{R}$, and the largest two such blocks differ in size by
more than $1,$ or

\item there exist $k\in\mathbb{N}$, $\theta_{1},\ldots,\theta_{k}\in
\lbrack0,2\pi)$ such that the direct sum of the non-splitting $m\times m$
Jordan blocks of $T/r\left(  T\right)  $ that have spectral radius $1$ is
similar to%
\[
J_{m}\left(  R_{\theta_{1}}\right)  \oplus\cdots\oplus J_{m}\left(
R_{\theta_{k}}\right)
\]
with $\theta_{1}/2\pi\notin sp_{\mathbb{Q}}\left\{  1,\theta_{2}/2\pi
,\ldots,\theta_{k}/2\pi\right\}  $.
\end{enumerate}
\end{theorem}

\begin{proof}
We know that if $r\left(  T\right)  =0,$ then $T$ is nilpotent, which, by
Corollary \ref{power}, implies $T$ is $\mathbb{R}$-orbit reflexive. Hence we
can assume that $r\left(  T\right)  >0.$ Replacing $T$ by $T/r\left(
T\right)  ,$ we can, and do, assume $r\left(  T\right)  =1.$

In the case where every Jordan block of $T$ with spectral radius $r\left(
T\right)  $ splits, the proof that $T$ is not $\mathbb{R}$-orbit reflexive is
equivalent to the condition in $\left(  1\right)  $ is exactly the same at the
proof of Theorem 7 in \cite{HIMY}.

Next suppose $T$ satisfies $\left(  2\right)  $. Then, as in the proof of
$\left(  1\right)  \Longrightarrow\left(  4\right)  $ in Lemma \ref{hard},
given $\alpha\in\lbrack0,2\pi),$ we can choose a sequence $\left\{
s_{d}\right\}  $ of positive integers converging to $\infty$ such that
$s_{d}-m+1$ is even for each $d\geq1$ and such that $R_{\theta_{1}}%
^{s_{d}-m+1}\rightarrow R_{\alpha}$ and $R_{\theta_{j}}^{s_{d}-m+1}\rightarrow
I$ for $2\leq j\leq k$. It follows that%
\[
\frac{1}{\binom{s_{d}}{m-1}}J_{m}^{s_{d}}\left(  R_{\theta_{1}}\right)
\rightarrow\left(
\begin{array}
[c]{cccc}%
0 & \cdots & 0 & R_{\alpha}\\
0 & 0 & \cdots & \vdots\\
\vdots & \vdots & \ddots & 0\\
0 & 0 & \cdots & 0
\end{array}
\right)
\]
and for any of the other splitting or non-splitting $m\times m$ Jordan block
$J$ with $r\left(  J\right)  =1,$ we have%
\[
\frac{1}{\binom{s_{d}}{m-1}}J^{s_{d}}\rightarrow\left(
\begin{array}
[c]{cccc}%
0 & \cdots & 0 & I\\
0 & 0 & \cdots & \vdots\\
\vdots & \vdots & \ddots & 0\\
0 & 0 & \cdots & 0
\end{array}
\right)  .
\]
For any block $J$ with $r\left(  J\right)  <1$ or with size smaller than
$m\times m$, we have%
\[
\frac{1}{\binom{s_{d}}{m-1}}J^{s_{d}}\rightarrow0.
\]
Arguing as in the proof of $\left(  1\right)  \Longrightarrow\left(  4\right)
$ in Lemma \ref{hard}, we see that, if $F$ is the flip matrix, and $S$ is the
matrix that is $\left(
\begin{array}
[c]{cccc}%
0 & \cdots & 0 & F\\
0 & 0 & \cdots & \vdots\\
\vdots & \vdots & \ddots & 0\\
0 & 0 & \cdots & 0
\end{array}
\right)  $ on the domain of $J_{m}\left(  R_{\theta_{1}}\right)  $, $\left(
\begin{array}
[c]{cccc}%
0 & \cdots & 0 & I\\
0 & 0 & \cdots & \vdots\\
\vdots & \vdots & \ddots & 0\\
0 & 0 & \cdots & 0
\end{array}
\right)  $ on the domains of each of the remaining $m\times m$ blocks $J$ with
$r\left(  J\right)  =1,$ and $0$ on the domains of the remaining blocks, then
$S\in\mathbb{R}$-$\mathrm{OrbRef}\left(  T\right)  ,$ but $ST\neq TS.$ Hence
$T$ is not $\mathbb{R}$-orbit reflexive.

We need to show that if $\left(  2\right)  $ holds with the condition on
$\theta_{1}$ replaced with condition $\left(  3\right)  $ in Lemma \ref{hard},
then $T$ must be $\mathbb{R}$-orbit reflexive. If $m=1$, then $T$ has the form
as in Lemma \ref{hard}, so we can assume that $m>1.$ Suppose $S\in\mathbb{R}%
$-$\mathrm{OrbRef}\left(  T\right)  $ and $0\neq\left(
\begin{array}
[c]{c}%
0\\
0\\
\vdots\\
x
\end{array}
\right)  =X$ is in the domain of $J_{m}\left(  R_{\theta_{1}}\right)  $. We
consider three cases.

\textbf{Case 1}. $S_{1}\left(  X\right)  =S\left(  X\right)  =0$, where
$S_{1}$ is the restriction of $S$ to the domain of $J_{m}\left(  R_{\theta
_{1}}\right)  $. Suppose $Y$ is orthogonal to the domain of $J_{m}\left(
R_{\theta_{1}}\right)  ,$ and using the fact that there is a sequence
$\left\{  m_{n}\right\}  $ of nonnegative integers and a sequence $\left\{
\lambda_{n}\right\}  $ in $\mathbb{R}$ such that
\[
S\left(  X+Y\right)  =\lim_{n\rightarrow\infty}\lambda_{n}T^{m_{n}}\left(
X+Y\right)  ,
\]
which means that
\[
0=S\left(  X\right)  =\lim_{n\rightarrow\infty}\lambda_{n}T^{m_{n}}\left(
X\right)  ,
\]
and
\[
S\left(  Y\right)  =\lim_{n\rightarrow\infty}\lambda_{n}T^{m_{n}}\left(
Y\right)  .
\]
However, the former implies
\[
\lim_{n\rightarrow\infty}\left\vert \lambda_{n}\right\vert \binom{m_{n}}%
{m-1}=0,
\]
which implies $S\left(  Y\right)  =0.$ If $k\geq2,$ then $S_{2}=0,$ where
$S_{2}$ is the restriction of $S$ to the domain of $J_{m}\left(  R_{\theta
_{2}}\right)  ,$ so the preceding arguments imply that $S_{1}=0;$ whence,
$S=0.$

We therefore suppose $k=1,$ and it follows from $\left(  3\right)  $ that
$\theta_{1}/2\pi\in\mathbb{Q}$, i.e., $\theta_{1}=2\pi p/q$ with $1\leq p<q$
relatively prime integers. We can identify $\mathbb{R}^{2}$ with $\mathbb{C}$,
and we can write $x=re^{\alpha}$ with $r>0.$ Since $S\left(  X\right)  =0,$ we
have $S\left(  \frac{1}{r}X\right)  =0,$ so we can assume $x=e^{i\alpha}$.
Then $\left\{  \lambda R_{\theta_{1}}^{s}x:\lambda\in\mathbb{R},1\leq s\leq
q\right\}  $ is the set of all complex numbers whose argument belongs to
$\left\{  \alpha+jp2\pi/q:1\leq j\leq q\right\}  +\pi\mathbb{Z}$. Choose
numbers $\beta$ and $\gamma$ with $\alpha<\beta<\gamma<\alpha+\pi/8$ such
that
\[
\left[  \left\{  \gamma+jp2\pi/q:1\leq j\leq q\right\}  +\pi\mathbb{Z}\right]
\cap\left[  \left\{  \beta+jp2\pi/q:1\leq j\leq q\right\}  +\pi\mathbb{Z}%
\right]  =\varnothing.
\]
Since the argument of $e^{i\alpha}+te^{i\gamma}$ ranges over $\left(
\alpha,\gamma\right)  $ as $t$ ranges over $\left(  0,\infty\right)  $, we can
chose $t>0$ so that the argument of $e^{i\alpha}+te^{i\gamma}$ is $\beta$. Now
let $W=\left(
\begin{array}
[c]{c}%
0\\
0\\
\vdots\\
te^{i\gamma}%
\end{array}
\right)  $ in the domain of $J_{m}\left(  R_{\theta_{1}}\right)  $. Then
$S\left(  X+W\right)  =SX+SW=SW.$ However, the nonzero coordinates of any
vector in the closure of $\mathbb{R}$-\textrm{Orb}$\left(  T\right)  \left(
X+W\right)  $ are all complex numbers with arguments in $\left\{
\gamma+jp2\pi/q:1\leq j\leq q\right\}  +\pi\mathbb{Z}$ and the nonzero
coordinates of any vector in the closure of $\mathbb{R}$-\textrm{Orb}$\left(
T\right)  \left(  X+W\right)  $ are all complex numbers with arguments in
$\left\{  \beta+jp2\pi/q:1\leq j\leq q\right\}  +\pi\mathbb{Z}$ Hence
$S_{1}\left(
\begin{array}
[c]{c}%
0\\
0\\
\vdots\\
y
\end{array}
\right)  =0$ for every choice of $y.$ We can apply similar arguments to each
of the other coordinates to get $S_{1}=0$, which implies $S=0.$

\textbf{Case 2}. $S\left(  X\right)  =S_{1`}\left(  X\right)  =\lambda
_{0}T^{n_{0}}\left(  X\right)  \neq0.$ Note that if $\lambda T^{s}\left(
X\right)  =\left(
\begin{array}
[c]{c}%
x_{1}\\
x_{2}\\
\vdots\\
x_{m}%
\end{array}
\right)  \neq0$, then%
\[
\frac{\left\Vert x_{m-1}\right\Vert }{\left\Vert x_{m}\right\Vert }=s,\text{
and }R_{\theta_{1}}^{-s}x_{m}=\lambda x.
\]
This means that if $\left\{  m_{n}\right\}  $ is a sequence of nonnegative
integers and $\left\{  \lambda_{n}\right\}  $ is a sequence in $\mathbb{R}$,
and $T^{m_{n}}\left(  X\right)  \rightarrow S\left(  X\right)  ,$ then,
eventually $m_{n}=n_{0}$ and $\lambda_{n}\rightarrow\lambda_{0}$. It follows
that $S=\lambda_{0}T^{n_{0}}$ on the orthogonal complement of the domain of
$S_{1}$. If $k\geq2,$ we can argue (using $S_{2}$) that $S=\lambda_{0}%
T^{n_{0}}.$ If $k=1,$ we can use $M_{1},M_{2},M_{3}$ as in Case 1 to show that
$S=\lambda_{0}T^{n_{0}}.$

\textbf{Case 3}. $S\left(  X\right)  =S_{1}\left(  X\right)  =\left(
\begin{array}
[c]{c}%
x_{1}\\
x_{2}\\
\vdots\\
x_{m}%
\end{array}
\right)  \neq0,$ but $x_{m}=0.$ If $\left\{  s_{n}\right\}  $ is a sequence of
nonnegative integers and $\left\{  \lambda_{n}\right\}  $ is a sequence in
$\mathbb{R}$ and $\lambda_{n}T^{s_{n}}\left(  X\right)  \rightarrow S\left(
X\right)  ,$ we must have $\lambda_{n}\rightarrow0,$ and thus $s_{n}%
\rightarrow\infty,$ and $\left\{  \left\vert \lambda_{n}\right\vert
\binom{s_{n}}{m-1}\right\}  $ bounded. Thus $S_{1}\left(  X\right)  =\left(
\begin{array}
[c]{c}%
x_{1}\\
0\\
\vdots\\
0
\end{array}
\right)  $. It follows that if $J$ is an $m\times m$ Jordan block of $T$ with
$r\left(  J\right)  =1$ and whose domain is orthogonal to the domain of
$S_{1}$, then the restriction of $S$ to the domain of $J$ is a matrix whose
only nonzero entry is in the first row and $m^{th}$ column. The restriction of
$S$ to the domain of a block $J$ with $r\left(  J\right)  <1$ or whose size is
smaller than $m\times m$ must be $0$. If $k\geq2$, the $S_{1}$ also has an
operator matrix whose only nonzero entry is in the first row and $m^{th}$
column. If $k=1,$ then $\theta_{1}/2\pi$ is rational, and we can argue with
$M_{1},M_{2},M_{3}$ as in Case 1 to see that $S_{1}$ has a matrix whose only
nonzero entry is in the first row and $m^{th}$ column. If the $m\times m$
Jordan blocks of $T$ are $J_{m}\left(  R_{\theta_{1}}\right)  \oplus
\cdots\oplus J_{m}\left(  R_{\theta_{k}}\right)  \oplus J_{m}\left(
I_{a}\right)  \oplus J_{m}\left(  -I_{b}\right)  $ ($I_{a}$ is an $a\times a$
identity matrix), then the corresponding decomposition of $S$ is a direct sum
of $\left(
\begin{array}
[c]{cccc}%
0 & \cdots & 0 & A_{j}\\
0 & 0 & \cdots & \vdots\\
\vdots & \vdots & \ddots & 0\\
0 & 0 & \cdots & 0
\end{array}
\right)  $, $1\leq j\leq k+2$. It is easily seen that $A_{1}\oplus\cdots\oplus
A_{k+2}$ is in $\mathbb{R}$-\textrm{OrbRef}$\left(  R_{\theta_{1}}\oplus
\cdots\oplus R_{\theta_{k}}\oplus I_{a}\oplus-I_{b}\right)  .$ Since
$\theta_{1},\ldots,\theta_{k}$ satisfy condition $\left(  3\right)  $ in Lemma
\ref{hard}, it follows from Lemma \ref{hard} that $R_{\theta_{1}}\oplus
\cdots\oplus R_{\theta_{k}}\oplus I_{a}\oplus-I_{b}$ is $\mathbb{R}$-orbit
reflexive, so there is a sequence $\left\{  s_{n}\right\}  $ with
$s_{n}\rightarrow\infty$ and a sequence $\left\{  \lambda_{n}\right\}  $ in
$\mathbb{R}$ such that $\lambda_{n}\left(  R_{\theta_{1}}\oplus\cdots\oplus
R_{\theta_{k}}\oplus I_{a}\oplus-I_{b}\right)  ^{s_{n}-m+1}\rightarrow
A_{1}\oplus\cdots\oplus A_{k+2}$. Hence%
\[
\lambda_{n}T^{s_{n}}\rightarrow S.
\]
Hence $T$ is $\mathbb{R}$-orbit reflexive.
\end{proof}

\bigskip

\begin{remark}
Using the ideas of the proof of Corollary \ref{coror} it is possible to
characterize $\mathbb{R}$-orbit reflexivity for an algebraic operator on a
Banach space in terms of its algebraic Jordan form.
\end{remark}

\end{document}